\documentclass[twoside,11pt,reqno]{amsart} 
\usepackage{amsmath,amssymb,amscd,mathrsfs,stmaryrd,wasysym,latexsym}

\makeatletter

\@addtoreset{equation}{section}

\newtheorem{Proposition}{Proposition}[section]
\newtheorem{Lemma}[Proposition]{Lemma}
\newtheorem{Theorem}[Proposition]{Theorem}
\newtheorem{Corollary}[Proposition]{Corollary}

\newtheorem{Remark}[Proposition]{Remark}

\newtheorem{Problem}[Proposition]{Problem}

\newbox\squ  
\setbox\squ=\hbox{\vrule width.6pt
   \vbox{\hrule height.6pt width.4em\kern1ex
         \hrule height.6pt}%
   \vrule width.6pt}

\def\F{\mathfrak{F}}

\def\cP{\mathcal{P}}
\def\cRP{\mathcal{RP}}

\def\gmod#1{\operatorname{Rep}(#1)}

\def\C{{\mathbb C}}
\def\bbF{{\mathbb F}}
\def\Q{{\mathbb Q}}
\def\Z{{\mathbb Z}}

\def\0{{\bar 0}}
\def\1{{\bar 1}}

\def\St{{\mathscr{T}}}
\def\T{{\mathtt T}}

\def\HOM{{\operatorname{HOM}}}

\def\res{{\operatorname{res}\,}}
\def\Res{{\operatorname{res}\:}}

\def\sh{{\operatorname{sh}}}
\def\cha{{\operatorname{char}\,}}

\def\chq{{\operatorname{ch}_q\:}}
\def\qdim{{\operatorname{qdim}\:}}

\def\bi{\text{\boldmath$i$}}
\def\bj{\text{\boldmath$j$}}

\def\bA{\text{\boldmath$A$}}
\def\bB{\text{\boldmath$B$}}

\def\phi{{\varphi}}
\def\emptyset{{\varnothing}}

\def\la{{\lambda}}
\def\La{{\Lambda}}

\def\underbar{\mathpalette\@underbar}
\def\@underbar#1#2{\settowidth{\@tempdimb}{$#1#2$}\@tempdimb=0.8\@tempdimb
                   \ooalign{$#1#2$\crcr%
                         \hfil\rule[-.5mm]{\@tempdimb}{.4pt}\hfil}}

\newdimen\hoogte    \hoogte=14pt    
\newdimen\breedte   \breedte=14pt   
\newdimen\dikte     \dikte=0.5pt    

\newenvironment{young}{\begingroup
       \def\vr{\vrule height0.8\hoogte width\dikte depth 0.2\hoogte}
       \def\fbox##1{\vbox{\offinterlineskip
                    \hrule height\dikte
                    \hbox to \breedte{\vr\hfill##1\hfill\vr}
                    \hrule height\dikte}}
       \vbox\bgroup \offinterlineskip \tabskip=-\dikte \lineskip=-\dikte
            \halign\bgroup &\fbox{##\unskip}\unskip  \crcr }
       {\egroup\egroup\endgroup}
\def\diagram#1{\relax\ifmmode\vcenter{\,\begin{young}#1\end{young}\,}\else%
              $\vcenter{\,\begin{young}#1\end{young}\,}$\fi}
              
\begin{document}

\title[An interpretation of the LLT algorithm]{An interpretation of the Lascoux-Leclerc-Thibon algorithm and graded representation theory}
\author{Alexander Kleshchev and David Nash}

\begin{abstract}
We use graded Specht modules  to calculate the graded decomposition numbers for the Iwahori-Hecke algebra of the symmetric group over a field of characteristic zero at a root of unity.  The algorithm arrived at is the Lascoux-Leclerc-Thibon algorithm in disguise.  Thus we interpret the algorithm in terms of graded representation theory.

\end{abstract}
\thanks{Supported in part by the NSF grant 
DMS-0654147. The paper was completed while the authors were visiting the Isaac Newton Institute for Mathematical Sciences in Cambridge, U.K., whom we thank for hospitality and support.}
\thanks{}
\address{Department of Mathematics, University of Oregon, Eugene, USA.}
\email{klesh@uoregon.edu}
\address{Department of Mathematics, University of Oregon, Eugene, USA.}
\email{dnash@uoregon.edu}
\maketitle

\section{Introduction}
Let $H_d$ be the Hecke algebra of the symmetric group $S_d$ defined over $\C$ with a parameter $\xi$ which is a primitive $e^\text{th}$ root of unity. The problem of finding  dimensions of the irreducible $H_d$-modules is equivalent to the problem of finding decomposition numbers for $H_d$. This problem has been solved by Ariki \cite{Ariki} who proved a conjecture of Lascoux, Leclerc and Thibon \cite{LLT} which gives an algorithm for computing the decomposition numbers. 

Recently, Brundan and the first author \cite{BK,BK2} have defined {\em graded decomposition numbers} for $H_d$ and proved a graded analogue of Ariki's theorem. The grading provides new information about $H_d$-modules, which is collected into a {\em graded character} of a module. Graded characters of Specht modules were computed in \cite{BKW}. This together with the main results of \cite{BK2} is sufficient to yield an algorithm for computing graded decomposition numbers for $H_d$. 

Despite the different approach, the algorithm is equivalent to the one suggested in \cite{LLT}, although this equivalence is not immediately obvious. In section~\ref{SLLT} we explain the equivalence of the two algorithms. Note that our approach gives a new interpretation of some coefficients computed in the LLT algorithm. 
We direct the reader to \cite{L} for examples of such computations.

\vspace{1 mm}
\noindent
{\bf Acknowledgement.} In the previous version of the paper we missed the equivalence of our algorithm and the LLT algorithm. We are grateful to the referee for pointing this out. 

\section{Preliminaries}

\subsection{Basic objects}
Let $\bbF$ be an algebraically closed field and $\xi \in \bbF^\times$.  The \emph{quantum characteristic}, $e$, is the smallest positive integer such that 
$1 + \xi + \xi^2 + \cdots + \xi^{e-1} = 0,$
where we set $e:=0$ if no such integer exists. Set 
$I:=\Z/e\Z.$  
For any $i \in I$ we have a well-defined element 
$$\nu(i):= \begin{cases}
i &  \text{if $\xi=1$ },\\
\xi^i &  \text{if $\xi \neq 1$ },\\
 \end{cases}$$
of $\bbF$. Throughout the paper, $q$ is an indeterminate and we define the \emph{bar-involution} on $\Z[q, q^{-1}]$ by $\overline{p(q)}=p(q^{-1})$ for all $p(q) \in \Z[q, q^{-1}]$.  We 
set
 $$[n]_q:=\frac{q^n-q^{-n}}{q-q^{-1}} \text{~and \hspace{.1cm}} [n]_q^!:=[n]_q[n-1]_q\cdots [1]_q.$$
 

The \emph{Iwahori-Hecke algebra} of $S_d$ with parameter $\xi$, is the $\bbF$-algebra  $H_d=H_d(\bbF, \xi)$ with generators
$T_1, T_2, \dots, T_{d-1}$
and relations
$$T_r^2 = (\xi - 1)T_r + \xi \:\: \text{~~~~~($1 \leq r <d$)},$$
$$T_rT_{r+1}T_r = T_{r+1}T_rT_{r+1} \: \text{~~($1 \leq r < d-1$)},$$
$$T_rT_s=T_sT_r \:\: \text{~~~~($1 \leq r,s <d$, $|r-s|>1$)}.$$
The group algebra $\bbF S_d$ of the symmetric group appears in the special case when $\xi=1$. In this case $e=\cha{\bbF}$. 

The \emph{Jucys-Murphy elements} of $H_d$ are: 
$$L_r = \begin{cases}
 (1,r) + (2,r) + \cdots + (r-1,r) &  \text{if $\xi = 1$ }\\
\xi^{1-r}T_{r-1}\cdots T_2T_1T_1T_2\cdots T_{r-1} &  \text{if $\xi \neq 1$}\\
 \end{cases}
 \quad (1\leq r\leq d).$$
It follows from \cite[Lemma 4.7]{G} and \cite[Lemma 7.1.2]{K} that the 
eigenvalues of $L_1, \dots, L_d$ on any finite dimensional $H_d$-module are of the form $\nu(i)$ for $i \in I$. 
For $\bi=(i_1,\dots,i_d) \in I^d$ and a finite dimensional $H_d$-module $V$, we define the \emph{$\bi$-weight space} of $V$ to be
$$V_\bi:=\{v \in V | (L_r - \nu(i_r))^N v = 0 \text{~for~} N\gg0 \text{~and~} r=1,\dots,d\}.$$
Then we have the \emph{weight space decomposition} 
$V= \oplus_{\bi \in I^d} V_\bi.$

\subsection{Partitions and tableaux} We use the partition notation from \cite{BK2}.  
In particular, $\cP_d$ is the set of all partitions of  $d$.  Given a partition $\lambda=(\lambda_1,\la_2,\dots) \in \cP_d$, we define
$\sigma_k(\lambda) = \sum_{i=1}^k \lambda_i.$ 
Denote the usual \emph{dominance order} on $\cP_d$ by `$\unrhd$', see \cite{JamesBook}.  
A partition $\lambda=(\lambda_1, \lambda_2, \dots)$ is called \emph{$e$-restricted} if $\lambda_r - \lambda_{r+1} < e$ for all $r=1,2,\dots$.  We let $\cRP_d \subset \cP_d$ be the subset of all $e$-restricted partitions of $d$.

Let $\la\in \cP_d$. A node $A\in\lambda$ is called \emph{removable $($for $\lambda)$} if $\lambda \setminus\{A\}$ has the shape of a partition.  A node $B {\not} {\in} \lambda$ is called \emph{addable $($for $\lambda)$} if $\lambda \cup \{B\}$ has the shape of a partition.  Given any set of removable nodes $\bA = \{A_1,\dots,A_m\}$ for $\lambda$, we denote $\lambda_{\bA}:=\lambda \setminus \{A_1,A_2, \dots, A_m\}$.

If the node $A$ is in row $a$ and column $b$, we write $A=(a,b)$, and then the {\em residue} of $A$ is defined to be  
$\Res A:=b-a\pmod{e}\in I.$ For $i\in I$, 
a  node $A$ is called an \emph{$i$-node} if $\Res A = i$.  
For $\lambda, \mu \in \cP_d$, we write $\lambda \sim \mu$ if and only if for each $i\in I$ the number of $i$-nodes in $\la$ is equal to that in $\mu$.

Given $\mu \in \cP_d$ we call $\lambda \in \cP_d$ a \emph{move for $\mu$} if $\lambda \unlhd \mu$ and $\lambda \sim \mu$.  We denote the set of moves for $\mu$ by $M(\mu)$. 
Put $M(\mu,\lambda):=\{\nu \in M(\mu) | \lambda \in M(\nu)\}$. Note that if $\la$ is not a move of $\mu$ then $M(\mu,\la)$ is empty. 
For $\lambda \in M(\mu)$, we define the \emph{distance} between $\lambda$ and $\mu$ to be
$$\l(\mu,\lambda):=\sum_{k \geq 1} \sigma_k(\mu)-\sigma_k(\lambda).$$
Since $\lambda \unlhd \mu$, we know that $\sigma_k(\lambda) \leq \sigma_k(\mu)$ for all $k \geq 1$.  So $l(\mu,\lambda) \geq 0$, with equality if and only if $\lambda=\mu$.  Moreover, if $\nu \in M(\mu,\lambda)$, then $l(\mu,\nu) \leq l(\mu,\lambda)$ with equality if and only if $\nu=\lambda$.  

Following \cite{James}, for $m\in\Z_{>0}$, we define the $m^\text{th}$ \emph{ladder} $L_m$ as the set of nodes of the form $(1+k, m-k(e-1))$ for all non-negative integers $k$ with $k<m/(e-1)$. Informally,  the ladders are straight lines with slope $1/e$. Note that our ladders are transposed to those of James, since we are using the newer Dipper-James-Mathas notation for Specht modules. 
The nodes in the ladder $L_m$ all have the residue $m-1\pmod{e}$, which we refer to as the {\em residue} of the ladder, and denote $\Res L_m$.
For a partition $\lambda$ and a positive integer $m$, we set $r_m(\lambda):=|\lambda \cap L_m|$.  Denote by $t_{\lambda}$ the maximal index such that $r_{t_{\lambda}}(\lambda)\neq 0$, and refer to the ladders $L_1,L_2,\dots,L_{t_{\lambda}}$ as the {\em ladders of $\lambda$} (some of them could have trivial intersection with $\la$).  
A ladder $L_m$ is \emph{bottom complete for $\lambda$} if whenever a node $A=(a,b)\in L_m$ belongs to $\la$, all other nodes $(a',b')\in L_m$ with $a'>a$ also belong to $\la$. 

\begin{Lemma}\label{BottomComp}\cite[1.2]{James} Let $\lambda \in \cP_d$.  Then $\lambda$ is $e$-restricted if and only if all ladders are bottom complete for $\lambda$.
\end{Lemma}

Given nodes $A=(a_1,a_2)$ and $B=(b_1,b_2)$ we say $A$ is \emph{above} (resp. \emph{below}) $B$ if $a_1<b_1$ (resp. $a_1>b_1$). Let $\la\in \cP_d$. For a removable $i$-node $A$ in $\la$ we define the \emph{degree of $A$} to be:
\begin{equation*}
\begin{split}
d_A(\lambda) := 
\#\{\text{addable $i$-nodes
below $A$}\}\\
-\#\{\text{removable $i$-nodes
below $A$}\}.
\end{split}
\end{equation*}

For $\lambda\in \cP_d$, let $\St(\lambda)$ denote the set of all standard $\lambda$-tableaux. 
Let $\T\in\St(\lambda)$.  For $s \in \Z_{\geq 0}$ we denote by $\T_{\leq s}$ and $\sh(T_{\leq s})$ the tableau obtained by retaining the nodes of $\T$ labeled by the numbers $1,\dots,s$ and its shape respectively, so that $\T_{\leq s}\in\St(\sh(T_{\leq s}))$.
Let $A$ be the node of $\lambda$ labeled by $d$ in $\T$. 
Following \cite{BKW}, we define the degree of $\T$ inductively by:
$$\deg(\T) = \begin{cases}
d_A(\lambda) + \deg(\T_{\leq d-1}) &  \text{if $d>0$},\\
0 &  \text{if $d=0$}.\\
\end{cases}$$
For $\T \in \St(\lambda)$ we also have the associated \emph{residue sequence}: 
$$\bi^\T=(i_1,\dots,i_d) \in I^d,$$
where $i_r$ is the residue of the node labeled by $r$ in $\T$ for $1 \leq r \leq d$.

\subsection{\boldmath Representation theory of $H_d$}\label{RepTheory}
From the work of Dipper, James and Mathas \cite{DJM}, the algebra $H_d$ has a special family of modules $\{S(\mu)\mid\mu\in\cP_d\}$,  labeled by the partitions of $d$, called \emph{Specht modules}. This goes back to Dipper-James \cite{DJ}, although the Specht modules defined in \cite{DJ} are different from those in \cite{DJM}. Here we follow the conventions of \cite{DJM}. 

If $e=0$, then $\{S(\mu) \mid \mu \in \cP_d\}$ is a complete irredundant set of the irreducible $H_d$-modules.  In the more interesting case $e>0$, the head $D(\mu)$ of $S(\mu)$ is irreducible, provided $\mu\in\cRP_d$, and  
$\{D(\mu)\mid\mu \in \cRP_d\}$
is a complete irredundant set of the irreducible $H_d$-modules.  

In this paper, grading always means $\Z$-grading. The algebra $H_d$ has an explicit grading exhibited in \cite{BK}. Therefore we may  speak of {\em graded} $H_d$-modules. In \cite[\S2.7]{BK2}, a notion of graded duality $\circledast$ on finite dimensional graded $H_d$-modules is introduced. It is then shown that each irreducible $H_d$-module $D(\mu)$ has a unique grading which makes it into a {\em graded} $H_d$-module with $D(\mu)^\circledast\cong D(\mu)$. Moreover, by \cite{BKW}, the Specht modules are also gradable as $H_d$-modules in such a way that the natural projection $S(\mu)\twoheadrightarrow D(\mu)$ is a degree zero map whenever $\mu \in \cRP_d$.

If $V=\bigoplus_{m\in \Z} V_m$ is a finite dimensional graded vector space, define $\qdim V:=\sum _{m\in \Z} (\dim V_m)q^m\in\Z[q,q^{-1}]$. 
Let $\mathscr{C}$ be the free $\Z[q,q^{-1}]$-module on $I^d$. Given a finite dimensional graded $H_d$-module $V$, we define its \emph{graded character} to be 
$$\chq V:=\sum_{\bi \in I^d} (\qdim V_\bi) \cdot \bi \in \mathscr{C}$$

The graded characters of Specht modules are now as follows:

\begin{Theorem}\label{Specht}{\rm \cite[\S4.3]{BKW}} Let $\mu\in \cP_d$. Then $\chq S(\mu) = \sum_{\T \in \St(\mu)}q^{\deg(\T)}\bi^{\T}.$
\end{Theorem}

In particular, $\chq S(\mu)$ depends only on  $\mu$ and $e$. 
We extend the bar-involution from $\Z[q,q^{-1}]$ to $\mathscr{C}$ so that 
$\bar{\bi}=\bi$ for all $\bi\in I^d$.  
The fact that $D(\mu)^\circledast\cong D(\mu)$ can now be interpreted as follows:  

\begin{Theorem}\label{Irred}{\rm\cite[Theorem 4.18(3)]{BK2}} Let $\lambda \in \cRP_d$. Then $\chq D(\la)$  is bar-invariant.
\end{Theorem}

When viewed as graded $H_d$-modules, the set $\{D(\mu) \mid \mu \in \cRP_d\}$ forms a complete irredundant set of the irreducible graded $H_d$-modules \emph{up to a grading shift}.  Given $m\in\Z$ and a graded $H_d$-module $V$, we let $V\langle m\rangle$ denote the module obtained by shifting the grading of  $V$ up by $m$.  Thus we have $\chq V\langle m\rangle=q^m\,\chq V$.

For $\mu \in \cP_d$ and $\lambda \in \cRP_d$, we define the corresponding \emph{graded decomposition number} as follows 
$$d_{\mu,\lambda}=d_{\mu,\la}(q):=\sum_{m\in\Z} a_mq^m \in \Z_{\geq 0}[q,q^{-1}],$$  where $a_m$ is the multiplicity of $D(\la)\langle m\rangle$ in a graded composition series of $S(\mu)$. 
Note that $d_{\mu,\la}(1)$ is the usual decomposition number, so the following result easily follows from the well-known facts in the ungraded setting and the fact that the natural map $S(\mu)\twoheadrightarrow D(\mu)$ is of degree zero.

\begin{Theorem}\label{OldFund} 
Let $\lambda \in \cRP_d$ and $\mu \in \cP_d$.  Then

{\rm (i)} $d_{\mu,\lambda}=0$ unless $\lambda\in M(\mu)$. 

{\rm (ii)} $d_{\lambda,\lambda}=1$.
\end{Theorem}

We have 
\begin{equation*}\label{EE}
\chq S(\mu) = \sum_{\lambda \in \cRP_d} d_{\mu,\lambda} \chq D(\lambda).
\end{equation*}
By Theorem~\ref{OldFund}, the graded decomposition matrix $(d_{\mu,\la})$ is unitriangular, so the knowledge of the graded decomposition numbers implies the knowledge of the graded characters of the irreducible $H_d$-modules. The converse is also true since the graded characters of the irreducible  $H_d$-modules are linearly independent, see e.g. \cite[Theorem 3.17]{KL1}).

The following key fact is special for the case $\cha \bbF=0$: 

\begin{Theorem}\label{NewFund} \cite[Theorem 3.9 and Corollary 5.15]{BK2}
Let $\cha \bbF=0$, $\lambda \in \cRP_d$ and $\mu \in \cP_d$. If $\mu\neq \la$, then $d_{\mu,\lambda} \in q\Z_{\geq 0}[q]$.
\end{Theorem}

\section{The ladder weight}\label{JLambdaSect}

\subsection{A dominance lemma}
For $\la\in \cRP_d$, let $t=t_\lambda$ be the index of its bottom ladder, and $r_t(\lambda)=|\la\cap L_t|$. Denote $\lambda \cap L_t = \{A_1,\dots,A_{r_t(\la)}\}$. Order the nodes of this set so that  $A_u$ is below $A_s$ whenever $u<s$, and refer to the sequence $(A_1,\dots,A_{r_t(\la)})$ as the \emph{bottom removable sequence of $\lambda$}. Observe that the nodes of the bottom removable sequence of $\la$ are indeed removable nodes of  $\lambda$. All nodes of this sequence have the same residue which we refer to as the {\em residue of the bottom removable sequence}. 

The following technical result generalizes  Lemma 1.1 in \cite{KS}.

\begin{Lemma}\label{Tech} Let $\lambda \in \cRP_d$ and $\mu \in \cP_d$, with $\lambda {\not} {\unlhd} \mu$.  Let $\bA=(A_1,\dots,A_r)$ be the bottom removable sequence for $\lambda$, and $i$ be its residue. If $\bB=\{B_1,\dots,B_r\}$ is any set of $r$ removable $i$-nodes for $\mu$ then $\lambda_{\bA} {\not} {\unlhd} \mu_{\bB}$
\end{Lemma}
\begin{proof} Let $\lambda_{\bA} \unlhd \mu_{\bB}$. We need to show that $\lambda \unlhd \mu$. 
Let $A_m$ be in row $j_m$ of $\lambda$, and let $B_m$ be in row $l_m$ of $\mu$ for $m=1,\dots,r$.  By our convention 
$j_m = j_1-(m-1)$ for $1 \leq m \leq r$.  We may also assume that $l_1>\dots>l_r$.  

Let $\lambda^m = \lambda \setminus \{A_{m+1},\dots,A_r\}$ and $\mu^m = \mu \setminus\{B_{m+1},\dots,B_r\}$ for $0 \leq m \leq r$.  Then it suffices to show by induction on $m=0,1,\dots,r$ that $\lambda^m\unlhd \mu^m$, with the induction base case, $m=0$, being our assumption.

Let $m>0$ and assume by induction that $\lambda^{m-1} \unlhd \mu^{m-1}$.  Note that 
$$\sigma_k(\lambda^m)=\begin{cases}
\sigma_k(\lambda^{m-1}) & \text{if $k<j_m$},\\
\sigma_k(\lambda^{m-1}) + 1 & \text{if $k \geq j_m$},\\
\end{cases}$$
$$
\sigma_k(\mu^m)=\begin{cases}
\sigma_k(\mu^{m-1}) & \text{if $k<l_m$},\\
\sigma_k(\mu^{m-1}) + 1 & \text{if $k \geq l_m$}.\\
\end{cases}$$
Since $\lambda^{m-1} \unlhd \mu^{m-1}$, we deduce that
\begin{equation*}\label{E0}
\sigma_k(\mu^m) \geq \sigma_k(\lambda^m) \text{\hspace{1cm}(for~ $k \geq l_m$ or $k < j_m$)}.
\end{equation*}

Observe that since $\bA$ was the bottom removable sequence, it follows that $j_1$ is the bottom non-empty row in $\lambda$.  Since $\lambda$ is $e$-restricted, we know also that $\res (j_1+1,1) \neq i$. Furthermore, we have that row $j_1+1$ is empty in $\lambda_\bA$, which implies that it is empty in $\mu_\bB$ as well since $\lambda_{\bA} \unlhd \mu_{\bB}$.

Since $B_1$ is an addable $i$-node for $\mu_\bB$, it follows that $l_1 < j_1+1$.  Since $j_m = j_1 - (m-1)$ and $l_n<l_{n-1}$ for all $1 \leq n \leq r$, it follows that $l_m \leq j_m$.
Thus for all $k$ we have either $k < j_m$ or $k \geq l_m$, which gives us that $\sigma_k(\mu^m) \geq \sigma_k(\lambda^m)$ for all $k>0$. Hence $\mu^m \unrhd \lambda^m$ completing the proof.
\end{proof}

\subsection{\boldmath Definition and properties of $\bj^\lambda$}
Let $\lambda \in \cRP_d$ have bottom removable sequence $\bA=(A_1,A_2,\dots,A_r)$ and define the {\em ladder weight $\bj^\la=(j_1,\dots,j_d)$} inductively as follows: $j_d=\res A_r$ and $(j_1,\dots,j_{d-1})=\bj^{\la_{A_r}}$. The idea of the ladder weight appears in \cite[\S6.2]{LLT}. 

\begin{Lemma}\label{JLambda} Let $\lambda \in \cRP_d$ and $\mu \in \cP_d$, with $\mu {\not} {\unrhd} \lambda$.  Then $\bj^{\lambda}$ does not appear in $\chq S(\mu)$. In particular, if $\mu$ is $e$-restricted, then $\bj^{\lambda}$ does not appear in $\chq D(\mu)$.
\end{Lemma}
\begin{proof} We apply induction on $d$, the base case $d=0$ being clear. Let $d>0$ and suppose for a contradiction that $\bi^{\T} = \bj^{\lambda}$ for some $\T \in \St(\mu)$.  Let $\bA = (A_1,\dots,A_r)$ be the bottom removable sequence of $\lambda$ and let $i$ be its residue.
Let $\bB=\{B_1,\dots,B_r\}$ be the nodes of $\mu$ labeled in $\T$ with $d,d-1,\dots,d-r+1$.  Since $\bi^{\T} = \bj^{\lambda}$ we have that $\Res B_1=\dots=\Res B_r = i$.
Let $\T'\in\St(\mu_{\bB})$ be the tableau obtained from $\T$ by removing  $B_1,\dots,B_r$.  Then $\bi^{\T'} = \bj^{\lambda_{\bA}}$, whence $\bj^{\lambda_{\bA}}$ appears in $\chq S(\mu_\bB)$.  By the inductive assumption,  $\mu_{\bB} \unrhd \lambda_{\bA}$. Now, by Lemma~\ref{Tech}, $\mu \unrhd \la$, a contradiction.
\end{proof}

\begin{Lemma}\label{K,K-Gen} Let $\lambda \in \cRP_d$ and set $t=t_\la$, $r_m=r_m(\la)$, $R_m:=r_1 + \dots + r_m$, and $\lambda(m):= \lambda \cap (L_1 \cup \dots \cup L_m)$ for $m>0$.  If $\T \in \St(\la)$ has $\bi^\T = \bj^\la$ then for each $m>0$ we have $\sh(T_{\leq R_m}) = \la(m)$.
\end{Lemma}
\begin{proof}  We apply induction on $m>0$ with the induction base $m=1$ being clear as $r_1=1$. Let $m>1$ and assume that $\sh(T_{\leq R_{m-1}}) = \lambda(m-1)$.  Letting $\bB$ denote the set of nodes in $\sh(T_{\leq R_m}) \setminus \sh(T_{\leq R_{m-1}})$, it suffices to prove that $\bB = \lambda \cap L_m$.  Since $|\bB|=|\lambda \cap L_m|=r_m$, it is enough to prove that $\bB \subseteq \lambda \cap L_m$.  Observing that $\bB$ is contained in $\lambda$ it then remains to show that $\bB \subseteq L_m$.
Observe that
the nodes of $\bB$ must have residue $\res L_m$ since $\bi^{\T} = \bj^{\lambda}$.  We know also that none of the nodes in $\bB$ belong to any of the ladders $L_1,\dots,L_{m-1}$ since $\sh(T_{\leq R_{m-1}}) = \lambda(m-1)$.  We conclude that $\bB \subseteq L_m$, completing the inductive step.
\end{proof}

Let $\lambda \in \cP_d$ and set $t=t_{\lambda}$ and $r_m=r_m(\lambda)$ for $m>0$. We define 
$$r_{\lambda}:=[r_1]_q^!\,  [r_2]_q^!\, \cdots\, [r_t]_q^!.$$

\begin{Lemma}\label{Coeff} If $\lambda \in \cRP_d$  then $\bj^{\lambda}$ has multiplicity $r_{\lambda}$ in $\chq S(\lambda)$.
 
\end{Lemma}
\begin{proof} 
Assume that $\bi^{\T} = \bj^{\lambda}$.   
Lemma~\ref{K,K-Gen} implies that for any tableau $\T \in \St(\lambda)$ with $\bi^{\T}=\bj^{\lambda}$, the numbers $d,d-1,\dots,d-r_t+1$ must appear in the bottom removable sequence $\bA=\{A_1,\dots,A_{r_t}\}$ for $\lambda$.  Moreover, all $r_t!$ possible permutations of those numbers occur.
Since $\lambda$ is $e$-restricted, by Lemma~\ref{BottomComp} all ladders are bottom complete for $\lambda$, and thus the possible labelings of $\bA$ give a contribution of exactly $[r_t]_q^!$ to the multiplicity of $\bj^{\lambda}$ in $\chq S(\lambda)$. The result follows by induction on $d$.
\end{proof}

\begin{Corollary}\label{CoeffCor} If $\lambda \in \cRP_d$  then $\bj^{\lambda}$ has multiplicity $r_{\lambda}$ in $\chq D(\lambda)$.
\end{Corollary}
\begin{proof}By Lemma~\ref{Coeff}, $\bj^{\lambda}$ appears in $\chq S(\lambda)$ with multiplicity $r_{\lambda}$, and by Lemma~\ref{JLambda}, $\bj^{\lambda}$ does not appear in $\chq D(\mu)$ for $\mu \lhd \lambda$.  Since composition factors of $S(\lambda)$ are of the form $D(\mu)$ for $\mu \unlhd \lambda$, the result follows.
\end{proof}

In two special cases we can be more explicit.

\begin{Corollary}\label{ColStd} Let $\lambda$ be a partition of $d$ with $\lambda_1<e$.  Then
$\bj^{\lambda}$ has multiplicity $1$ in $\chq S(\lambda)$ and in $\chq D(\lambda)$.
\end{Corollary}
\begin{proof}  Observe that $\lambda_1<e$ implies that $r_m \leq 1$ for all $1 \leq m \leq s_{\lambda}$.  The result then follows directly from Lemma~\ref{Coeff} and Corollary~\ref{CoeffCor}.
\end{proof}

\begin{Corollary}\label{ColStdHtP} Let $\lambda = (l_1^{k_1},l_2^{k_2},\dots)$ be an $e$-restricted partition with $l_1=e$ and $l_i>l_{i+1}$ for $i\geq 1$.  Then the multiplicity of $\bj^{\lambda}$ in $\chq S(\lambda)$ and in $\chq D(\lambda)$ is  
$([2]_q)^{k_1}.$
\end{Corollary}
\begin{proof} Use Lemma~\ref{Coeff} and Corollary~\ref{CoeffCor} and the fact that $\lambda$ has exactly $k_1$ ladders of size $2$, the remaining ladders of $\la$ having size at most $1$. 
\end{proof}

\subsection{Ladder weight multiplicity}
Let $V$ be a finite dimensional graded $H_d$-module. For any
$\lambda \in \cRP_d$ define $m_{\lambda}(V)$ to be the multiplicity of $\bj^{\lambda}$ in $\chq V$.  We collect the important properties of the function $m_{\lambda}$:

\begin{Theorem}\label{MLambda}  Let $\lambda \in \cRP_d$, $\mu \in \cP_d$, and $V$ be a finite dimensional graded $H_d$-module.  Then:

{\rm (i)} $m_{\lambda}(V) \in \Z_{\geq 0}[q,q^{-1}]$;

{\rm (ii)} if $m_{\lambda}(V) = 0$ then $[V:D(\lambda)]=0$;

{\rm (iii)} $m_{\lambda}(S(\lambda))=m_{\lambda}(D(\lambda))=r_{\lambda}$;

{\rm (iv)} $m_{\lambda}(S(\mu))=0$ unless $\lambda \in M(\mu)$; 

{\rm (v)} if $\mu\in \cRP_d$, then $m_{\lambda}(D(\mu))=0$ unless $\lambda \in M(\mu)$;

{\rm (vi)} $\displaystyle m_{\lambda}(S(\mu)) = \sum_{\nu \in \cRP_d \cap M(\mu,\lambda),\ \nu \neq \lambda}{d_{\mu,\nu} m_{\lambda}(D(\nu))} + d_{\mu,\lambda} r_{\lambda}$.
\end{Theorem}
\begin{proof} (i) is clear from the definitions, and (iii) is a restatement of Lemma~\ref{Coeff}  and Corollary~\ref{CoeffCor}. Next, (ii) follows from (i) and (iii). 

To see (iv), assume that $\la\not\in M(\mu)$. Then either $\la\not\sim\mu$ or $\la{\not}{\unlhd}\mu$. In the first case it follows from Theorem~\ref{OldFund} that $m_\la(S(\mu))=0$, and in the second case the same follows from Lemma~\ref{JLambda}. Now (v) follows from (iv) and (i). 

Finally, 
(vi) follows from (i), (iii), (v), and Theorem~\ref{OldFund}
\end{proof}

\section{The algorithm}

 \subsection{Basic Algorithm}\label{BasicAlg}
We give an easy algorithm to solve the following problem which will arise in the Main Algorithm below:

\begin{Problem}\label{PBasicAlg} Suppose $d(q) \in q \Z[q]$, and $m(q),r(q) \in \Z[q,q^{-1}]$ are such that $\overline{m(q)}=m(q)$, $\overline{r(q)}=r(q)$, and $r(q) \neq 0$.  If $d(q)r(q) + m(q)$ is known and $r(q)$ is known, find $d(q)$ and $m(q)$.
\end{Problem}

\begin{Remark}\label{RBasicAlg}{\rm It is easy to see that Problem~\ref{PBasicAlg} has a unique solution.}
\end{Remark}

We now explain an inductive algorithm to solve Problem~\ref{PBasicAlg}.  Clearly, if $d(q)r(q) + m(q)=0$, then $d(q)=m(q)=0$ by uniqueness.
If $d(q)r(q) + m(q)\neq 0$ then we can write it in the form 
$$d(q)r(q) + m(q)=\sum_{n=-N}^M a_nq^n \quad (\text{for}\ -N\leq M \text{~and~} a_{-N}\neq0, a_M\neq 0).$$
Note that the assumptions 
imply that $M\geq 0$ and $M \geq N$ (but we might have $N<0$ if $m(q)=0$).  
The algorithm proceeds by induction on the pairs of non-negative integers $(M, M+N)$ ordered lexicographically. 
The induction base is the pair $(0,0)$ where $d(q)=0$ and $m(q)=a_0$ by uniqueness.
Let $(M,M+N)>(0,0)$. This implies $M>0$.  We denote the top term of $r(q)$ by $bq^R$, and note that $R \geq 0$ since $r(q)$ is bar-invariant.   
We now consider two cases.

\emph{Case 1:} $M > N$. As $m(q)$ is bar-invariant, the term $a_Mq^M$ must come from $d(q)r(q)$.  Thus $\frac{a_M}{b}q^{M-R}$ is a term in $d(q)$.  Setting $d'(q):=d(q)-\frac{a_M}{b}q^{M-R}$, we are reduced to solving the problem for $d'(q)r(q) + m(q)$ for which $M'<M$.

\emph{Case 2:} $M = N$. As $d(q) \in q \Z[q]$, the term $a_{-N}q^{-N}$ must therefore come from $m(q)$.  Since $m(q)$ is bar-invariant, $a_{-N}q^N$ must be a term in $m(q)$ also.  Setting $m'(q):=m(q)-(a_{-N}q^{-N} + a_{-N}q^N)$, we are reduced to solving the problem for $d(q)r(q) + m'(q)$ for which $M' \leq M$ and $M'+N' < M+N$.

\subsection{Main Algorithm}\label{TheAlg}
From now on we assume that $\cha \bbF = 0$.  If $e=0$, then the Specht modules are irreducible, so we also assume that we are in the interesting case $e>0$, i.e we deal with the case of the Hecke algebra over a field of characteristic zero with parameter a primitive $e^\text{th}$ root of unity.  Under these assumptions, we will describe an algorithm for computing the graded decomposition numbers.

The algorithm relies heavily on Theorem~\ref{NewFund}, which is why we need the assumption $\cha \bbF = 0$. 

Let $\mu \in \cP_d$ and $\lambda \in \cRP_d$. We will compute $d_{\mu,\lambda}$ by induction.  However, this induction requires us to keep track of some extra information. 
By Theorem~\ref{OldFund}(i), $\lambda \not \in M(\mu)$ implies $d_{\mu,\lambda}=0$, so we assume $\lambda \in M(\mu)$.  We now calculate $d_{\mu,\lambda}$ and $m_{\lambda}(D(\mu))$ by induction on the distance $l(\mu,\lambda)$ (of course, $m_{\lambda}(D(\mu))$ only makes sense when $\mu$ is $e$-restricted).  Induction begins when $l(\mu,\la)=0$, hence $\mu=\lambda$, and we have  $d_{\mu,\mu}=1$ by Theorem~\ref{OldFund}(ii) and $m_{\mu}(S(\mu))=r_{\mu}$ by Theorem~\ref{MLambda}.

Let $l(\mu,\la)>0$, so $\mu \rhd \lambda$.  By induction, we know the graded decomposition numbers $d_{\mu,\nu}$ for all $\nu \in \cRP_d \cap M(\mu,\lambda)$ with $\nu \neq \lambda$ and the multiplicities $m_{\lambda}(D(\nu))$ for all $\nu \in \cRP_d \cap M(\mu,\lambda)$ with $\nu \neq \mu$.  To make the inductive step we need to compute $d_{\mu,\lambda}$ and, if $\mu$ is $e$-restricted, $m_{\lambda}(D(\mu))$.

If $\mu$ is not $e$-restricted, then by Theorem~\ref{MLambda}(vi), we have
$$ d_{\mu,\lambda}=\frac{1}{r_{\lambda}}\left(m_{\lambda}(S(\mu)) - \sum_{\nu \in \cRP_d \cap M(\mu,\lambda),\ \nu \neq \lambda}{d_{\mu,\nu} m_{\lambda}(D(\nu))}\right),$$
where all the terms in the right hand side are known by induction and Theorem~\ref{Specht}.

Let $\mu$ be $e$-restricted.  By Theorem~\ref{MLambda}(vi), we have 
$$m_{\lambda}(D(\mu)) + d_{\mu,\lambda} r_{\lambda} = m_{\lambda}(S(\mu)) - 
\sum_{\nu \in \cRP_d \cap M(\mu,\lambda),\ \nu \neq \lambda,\ \nu \neq \mu}{d_{\mu,\nu} m_{\lambda}(D(\nu))},$$
where all terms in the right hand side are known by induction and Theorem~\ref{Specht}.
Note $r_{\lambda}$ is non-zero and bar-invariant, $d_{\mu,\lambda} \in q \Z_{\geq 0}[q]$ by Theorem~\ref{NewFund},  and $m_{\lambda}(D(\mu))$ is bar-invariant by Theorem~\ref{Irred}.  Hence, we are in the assumptions of Problem~\ref{PBasicAlg} with $m(q)= m_{\lambda}(D(\mu))$, $d(q)=d_{\mu,\lambda}$, and $r(q)=r_{\lambda}$.  Now we apply the Basic Algorithm described in Section~\ref{BasicAlg} to calculate $m_{\lambda}(D(\mu))$ and $d_{\mu,\lambda}$ and complete the inductive step.

\begin{Remark}{\rm In our algorithm, to calculate $d_{\mu,\lambda}$ and $m_\la(D(\mu))$ one only ever needs to compute $d_{\nu,\kappa}$ and $m_{\kappa}(D(\nu))$ for pairs $(\nu,\kappa)$ such that $l(\nu,\kappa) < l(\mu,\lambda)$ and $\nu,\kappa \in M(\mu,\lambda)$.}
\end{Remark}

\section{Connection to LLT}\label{SLLT}

\subsection{Grothendieck groups} 
Here we recast some notions from \cite{LLT} and \cite{BK2}.  Let $[\gmod{H_d}]$ denote the Grothendieck group of the category of finite dimensional graded $H_d$-modules.  This is a free $\Z[q,q^{-1}]$-module with basis $\{[D(\la)] \mid \la \in \cRP_d\}$. 
For each $\la\in \cRP_d$, let $P(\la)$ be the projective cover of $D(\la)$; in particular there exists a degree preserving surjection $P(\la) \twoheadrightarrow D(\la)$.  

By definition of the graded decomposition numbers $d_{\mu,\la}$, for every $\mu\in \cP_d$, we have $[S(\mu)] = \sum_{\la \in \cRP_d} d_{\mu,\la} [D(\la)]$  in $[\gmod{H_d}]$. Moreover, 
from \cite[Theorem 3.14, Theorem 5.13]{BK2} it follows that in $[\gmod{H_d}]$ we have:
\begin{equation}\label{E1}
[P(\la)] = \sum_{\mu \in \cP_d} d_{\mu,\la} [S(\mu)]\qquad(\la \in \cRP_d).
\end{equation}

The \emph{Fock space} $\F$ is a $U_q(\hat{\mathfrak{sl}_e})$-module with $\Q(q)$-basis $\{\mu \mid \mu \in \bigoplus_{d \geq 0} \cP_d\}$.  The submodule of $\F$  generated by the vector $\emptyset$ (corresponding to the empty partition) is the irreducible highest weight module $V(\Lambda_0)$. There is a canonical $U_q(\hat{\mathfrak{sl}_e})$-module homomorphism $\pi:\F \twoheadrightarrow V(\Lambda_0)$, see \cite[(3.29)]{BK2}. We can and always will identify $\bigoplus_{d\geq 0}[\gmod{H_d}]\otimes_{\Z[q,q^{-1}]}\Q(q)=V(\La_0)\subset \F$.  Under this identification, we have  $\pi(\mu) = [S(\mu)]$ for each $\mu$.

\subsection{Projective $H_d$-modules}
For each $\bi\in I^d$ there is a unique idempotent $e(\bi)\in H_d$ (possibly zero) such that 
$e(\bi)V=V_\bi$ for any finite dimensional $H_d$-module, see e.g. \cite{BK3}.

\begin{Lemma}\label{He} If $\la \in \cRP_d$, then $\displaystyle H_d e(\bj^\la)=\bigoplus_{\mu \in \cRP_d} m_\la(D(\mu)) P(\mu)$.
\end{Lemma}
\begin{proof} If $M$ and $N$ are graded $H_d$-modules, we denote by $\HOM_{H_d}(M,N)$ the graded vector space which consists of {\em all}, and not necessarily homogeneous, $H_d$-homomorphisms from $M$ to $N$. Now, the graded multiplicity of $P(\mu)$ in $H_d e(\bj^\la)$ is equal to 
$$\qdim \HOM_{H_d}(H_d e(\bj^\la), D(\mu))=\qdim e(\bj^\la) D(\mu)= m_\la(D(\mu)),$$ as required.
\end{proof}

\begin{Corollary}\label{CHe} For $\la \in \cRP_d$, we have in $[\gmod{H_d}]$:

{\rm(i)} $[H_d e(\bj^\la)] = \sum_{\mu \in \cRP_d} m_\la(D(\mu)) [P(\mu)]$,

{\rm(ii)} $[H_d e(\bj^\la)] = \sum_{\mu \in \cP_d} m_\la(S(\mu)) [S(\mu)]$,

{\rm(iii)} $[H_d e(\bj^\la)] = \sum_{\mu \in \cRP_d} m_\la(P(\mu)) [D(\mu)]$.
\end{Corollary}
\begin{proof} (i) is a restatement of Lemma~\ref{He} in the Grothendieck group.  (ii) follows from (i) and (\ref{E1}) since:

$$\sum_{\mu \in \cRP_d} m_\la(D(\mu)) [P(\mu)] =\sum_{\mu \in \cRP_d} m_\la(D(\mu)) \sum_{\nu \in \cP_d} d_{\nu,\mu} [S(\nu)]$$
$$= \sum_{\nu \in \cP_d} \left(\sum_{\mu \in \cRP_d}m_\la(D(\mu)) d_{\nu,\mu} \right) [S(\nu)] = \sum_{\nu \in \cP_d} m_\la(S(\nu)) [S(\nu)].$$
Similarly, (iii) follows from (ii) and (\ref{E1}) since:
$$\sum_{\mu \in \cP_d} m_\la(S(\mu)) [S(\mu)] =\sum_{\mu \in \cP_d} m_\la(S(\mu)) \sum_{\nu \in \cRP_d} d_{\mu,\nu} [D(\nu)]$$
$$= \sum_{\nu \in \cRP_d} \left(\sum_{\mu \in \cP_d}m_\la(S(\mu)) d_{\mu,\nu} \right) [D(\nu)] = \sum_{\nu \in \cRP_d} m_\la(P(\nu)) [S(\nu)].$$
This completes the proof.
\end{proof}


The element $[H_de(\bj^\la)]$ is connected to the `first approximation' $A(\la)$ from  \cite{LLT} as follows. It is easy to check that $A(\la)=\frac{1}{r_\la}\sum_{\mu \in \cP_d} m_\la(S(\mu)) \mu \in \F$, and thus $\pi(r_\la A(\la)) = [H_de(\bj^\la)]$ using Corollary~\ref{CHe}(ii).  This gives an explicit connection between the formal sum $A(\la) \in \F$  and the representation theory of $H_d$.


\subsection{Comparing the algorithms}
Consider the following system of equations over $\Z[q,q^{-1}]$: 
\begin{equation}\label{system}
m_\la(S(\mu)) = \sum_{\nu \in \cRP_d} d_{\mu,\nu} m_\la(D(\nu))  \quad (\la \in \cRP_d, \: \mu \in \cP_d),
\end{equation}
with unknowns $d_{\mu,\nu}$ and $m_\la(D(\nu))$ for $\mu \in \cP_d$ and $\la,\nu \in \cRP_d$.  Note that $m_\la(S(\mu))$ are known from Theorem~\ref{Specht}.

The algorithm described in section~\ref{TheAlg} allows us to solve the system~(\ref{system}) and relies on the fact that it has a unique solution under the following conditions:
\begin{enumerate}
\item[{\rm (i)}] all $m_\la(D(\nu)) \in \Z[q,q^{-1}]$ are bar-invariant and $m_\la(D(\la))=r_\la$;

\item[{\rm (ii)}] $d_{\mu,\nu} = 0$ unless $\nu \unlhd \mu$, $d_{\nu,\nu}=1$, and $d_{\mu,\nu} \in q\Z[q]$ for $\mu \neq \nu$.
\end{enumerate}

It turns out that the LLT algorithm also relies on being able to solve the system~(\ref{system}).  In \cite{LLT}, the vector $A(\la)$ for $\la \in \cRP_d$ is a first approximation to the canonical basis element $G(\la)$, which calculates a \emph{column} of the decomposition matrix: $G(\la) = \sum_{\mu \in \cP_d} d_{\mu,\la} \mu$. 
One may write $r_\la A(\la)=\sum_{\nu \in \cRP_d} b_{\la,\nu} G(\nu)$ for some bar-invariant coefficients $b_{\la,\nu}$. 
Then:
$$\sum_{\mu \in \cP_d} m_\la(S(\mu)) \mu = \sum_{\nu \in \cRP_d} b_{\la,\nu}\sum_{\mu \in \cP_d} d_{\mu,\nu} \mu.$$
Fixing $\mu$ (i.e. fixing a \emph{row} of the matrix) we are left with solving
$$m_\la(S(\mu)) = \sum_{\nu \in \cRP_d} b_{\la,\nu}d_{\mu,\nu}$$
for each $\mu \in \cP_d$ and $\la \in \cRP_d$.  Since conditions analogous to (i) and (ii) are known to hold in the Lascoux-Leclerc-Thibon setup \cite{LLT}, we are left with solving the same system of equations under the same conditions.  Moreover, $b_{\la,\nu}$ appearing in the LLT algorithm can now be interpreted as the graded weight space multiplicities $m_\la(D(\nu))$ for each $\la, \nu \in \cRP_d$. 

\end{document}